\documentclass[12pt]{amsart}
\usepackage{amssymb}
\usepackage{enumerate}
\usepackage[all]{xy}
\title{Super Morita theory}
\author{Stephen Kwok}

\begin{document}

\theoremstyle{plain} \newtheorem{thm}{Theorem}[subsection]
\theoremstyle{plain} \newtheorem*{mythm}{Theorem}
\theoremstyle{plain} \newtheorem{lem}[thm]{Lemma}
\theoremstyle{plain} \newtheorem{prop}[thm]{Proposition}
\theoremstyle{plain} \newtheorem{cor}[thm]{Corollary}
\theoremstyle{definition} \newtheorem{defn}[thm]{Definition}

\maketitle

\begin{abstract}
We develop the basics of Morita theory for super rings. As an application, we produce a more explicit super Morita equivalence in the case of super Azumaya algebras.
\end{abstract}

\section{Introduction}

Two associative rings $R$ and $S$ with unit are
said to be {\it Morita equivalent} if the categories
$\mathfrak{M}_R$ and $\mathfrak{M}_S$ of right $R$- (resp. $S$-)
modules are equivalent.

The simplest and most prototypical example of Morita equivalence is
the equivalence between the category of $R$-modules and the category
of $\mathbb{M}_n(R)$-modules, where $\mathbb{M}_n(R)$ denotes the
ring of $n \times n$ matrices with entries in some associative ring
$R$. (This example is treated in detail in Ch. 17 of
\cite{Lam}).

Many interesting properties of rings (such as centers, K-theory, G-theory etc.) remain invariant under Morita equivalence, making the theory of Morita
equivalences a central tool in many areas of algebra.

The purpose of this work is to extend the basic theory of Morita
equivalences to super rings, i.e. rings $R$ with a
$\mathbb{Z}_2$-grading such that the multiplication is compatible
with the grading. The category of $R$-supermodules is that whose
objects are $\mathbb{Z}_2$-graded $R$-modules, and whose morphisms
are parity-preserving $R$-module homomorphisms.

We now briefly outline the contents of this paper. In the second section, we explain the supermodule theory needed for
the remainder of the paper, including the definitions of projective
module and generator.

The third section is the main portion of this paper and it contains
the proofs of the basic Morita theorems for supermodules. Our
treatment is modeled on the discussion of ungraded Morita theory in Ch. 18 of
\cite{Lam}; one may see Ch. 1 of \cite{B} for a more abstract formulation and development of this theory.

In the final section, we apply the super Morita theory developed in
the third section to the context of {\it super Azumaya algebras},
and show that for a super Azumaya algebra $A$, the Morita
equivalence may be realized more explicitly, in terms of the
{\it supercommutant} of an $(A, A)$-bimodule $M$.

In the ungraded case, the corresponding results on ordinary Azumaya
algebras may be found in III.5 of \cite{KnO}, where they are derived as a
corollary of more general considerations. However, we have preferred
to give a more self-contained and concrete treatment of this
material, better suited to our own purposes.

In another paper \cite{Kwo}, we develop the theory of
$\Pi$-invertible sheaves in supergeometry. Our main motivation for
the present work was ultimately to explain the existence of a
``product" structure on $\Pi$-invertible sheaves in terms
of the ``super skew field": the super Azumaya algebra $\mathbb{D} =
k [\theta]$, where $\theta$ is odd, $\theta^2 = -1$, for $k$ an
algebraically closed field of characteristic $\neq 2$.

We note that similar considerations are briefly discussed in the paper \cite{KKT}, in the context of quiver Hecke superalgebras. \\

\section {Module theory for super rings}

In this section $R$ denotes a (not necessarily commutative)
associative super ring with unit, that is, a $\mathbb{Z}_2$-graded
ring $R = R_0 \oplus R_1$ whose multiplication is compatible with
the grading:

\begin{align*}
R_i \cdot R_j \subseteq R_{i+j}
\end{align*}

\noindent where $i, j \in \mathbb{Z}_2$.

A {\it supermodule} for a super ring $R$ is a $\mathbb{Z}_2$-graded
$R$-module $M = M_0 \oplus M_1$ such that the $R$-module structure is compatible with
the grading on $M$:

\begin{align*}
R_i \cdot M_j \subseteq M_{i+j}
\end{align*}

The {\it parity} of a homogeneous element of a super ring $R$, or $R$-supermodule $M$, is defined as:

\begin{align*}
|x| = 0 \text{ if $x$ is even}\\
1 \text{ if $x$ is odd}
\end{align*}

$R$-supermodule homomorphisms are simply $R$-module homomorphisms
which preserve the grading (i.e., the parity). The set of
$R$-supermodule homomorphisms between supermodules $M$ and $N$ is
denoted $Hom_R(M,N)$ (for simplicity of notation, we will often drop
the subscript $R$ and write $Hom(M, N)$ when the ring $R$ under
consideration is clear). We will often refer to $Hom(M, N)$ as the
{\it categorical $Hom$}.

We will also have occasion to consider the collection of all
$R$-module homomorphisms between $M$ and $N$ (regarding both as {\it
ungraded} modules), which we denote by $\underline{Hom}_R(M, N)$. We
refer to $\underline{Hom}_R(-, -)$ as the {\it internal
$\underline{Hom}$}. As before, we will often drop the subscript $R$
when the ring under consideration is clear.

$\underline{Hom}(M, N)$ is endowed with a natural $\mathbb{Z}_2$-grading: the even part $(\underline{Hom}(M, N))_0$ consists of the parity-preserving homomorphisms, and the odd $(\underline{Hom}(M, N))_1$ consists of the parity-reversing homomorphisms. (The even homomorphisms are precisely the $R$-supermodule homomorphisms in the sense just defined).

We may make the collection of right (resp. left) $R$-modules into a
category $\mathfrak{M}_R$ (resp. ${}_R \mathfrak{M}$) by taking the
morphisms between two $R$-modules $M$ and $N$ to be the
$R$-supermodule homomorphisms $Hom(M, N)$. This is the reason we
refer to $Hom(-, -)$ as the categorical $Hom$.

The word ``homomorphism" may denote either elements of $Hom$
or of $\underline{Hom}$, but we will reserve the term {\it morphism}
solely for elements of $Hom$. Likewise, the terms {\it monomorphism,
epimorphism, isomorphism}, etc. will be understood to refer only to
elements of $Hom$ unless otherwise specified.

From now on we will drop the prefix ``super" and refer to objects of $\mathfrak{M}_R$ and ${}_R \mathfrak{M}$ simply as {\it modules}.

Some conventions: homomorphisms $f: M \to N$ of left $R$-modules
will occasionally be written to act on the right of $M$: thus we
will sometimes write $xf$ for $f(x)$. In this case we will use
$\bullet$ to denote composition of homomorphisms: $x (g \bullet f)$
for $f(g(x))$.

Note that with this convention, we have:

\begin{equation*}
(rx)f = (-1)^{|r||f|}r(xf)
\end{equation*}\\

Given super rings $S$, $R$, we may define the category of $(S, R)$-super bimodules ${}_S \mathfrak{M}_R$ in the obvious way (the corresponding definitions of categorical $Hom$ and internal $\underline{Hom}$ should by now be clear to the reader).

We define the dual of a right (resp. left) $R$-module $M$, denoted
$M^*$, to be the left (resp. right) $R$-module $\underline{Hom}(M, R)$, with the $R$-module structure given by left (resp. right) multiplication in $R$:

\begin{align*}
(rf)(x) &:= r[f(x)]\\
x(fr) &:= (-1)^{|x||r|}(xf)r
\end{align*}

Finally, if $M$ is a right $R$-module and $N$ is a left $R$-module,
the {tensor product} $M \otimes_R N$ is defined to be the usual
tensor product of $M$ and $N$, but endowed with the grading

\begin{align*}
&(M \otimes_R N)_k = \bigoplus_{i+j = k} M_i \otimes_R N_j\\
\end{align*}

If $f: M \to M', g: N \to N'$ are right (resp left) $R$-module
homomorphisms, then the induced map $f \otimes g: M \otimes_R N \to
M' \otimes N'$ is:

\begin{equation*}
(f \otimes g)(m \otimes n) := (-1)^{|g||m|} fm \otimes ng
\end{equation*}\\

We define the category of {\it superabelian groups} to be the
category of $\mathbb{Z}$-supermodules, with $\mathbb{Z}$ considered
as a (purely even) supercommutative ring: that is, the objects of
the category are $\mathbb{Z}_2$-graded abelian groups, with the
morphisms being the parity-preserving group homomorphisms. (There
should be no confusion with the completely different term
``supergroup"). The internal $\underline{Hom}$ of the category of
superabelian groups consists of all group homomorphisms, and we
denote the category of superabelian groups by $(SAb)$.

Using this terminology, we note that $M \otimes_R N$ has {\it a
priori} only the structure of a superabelian group.

If we suppose in addition that $M, M'$ are $(S,
R)$-bimodules, $N, N'$ are $(R, T)$-bimodules for $S, R$ super
rings, and $f, g$ are $(S, R)$ (resp. $(R, T)$-bimodule
homomorphisms, then one may verify that $f \otimes g: M \otimes_R N \to M' \otimes_R N'$ defined as above is an $(S, T)$-bimodule homomorphism.

Next we discuss the {\it opposite} of a super ring $R$. This is the super ring $R^o$ whose
underlying set is equal to that of $R$, and whose multiplication is
given by:

\begin{equation*}
x \cdot_{o} y := (-1)^{|x||y|} y \cdot x
\end{equation*}\

\noindent for $x, y \in R$, where the multiplication on the right hand side is that of $R$.

A left (resp. right) $R$-module $M$ may be canonically converted into a right (resp. left) $R^o$-module by defining the right $R$- (resp. left $R^o$)-action to be:

\begin{align*}
m \cdot_o r := (-1)^{|r||m|} r \cdot m
\end{align*}

\noindent (resp. \begin{align*}
r \cdot_o m := (-1)^{|r||m|} m \cdot r).\\
\end{align*}

One checks readily that this recipe sends left (resp. right) $R$-module homomorphisms to right (resp. left) $R^o$-module homomorphisms, and that it preserves the parity of homomorphisms. In particular, it defines a natural functor $\cdot {}^o: {}_R\mathfrak{M} \to \mathfrak{M}_{R^o}$ (resp. $\mathfrak{M}_R \to {}_{R^o} \mathfrak{M}$). The reader may verify that $(\cdot {}^o)^2$ is the identity functor (the key point is that $(R^o)^o = R$ as super rings), hence $\cdot {}^o$ is an equivalence of categories. We may therefore convert any statement about right (resp. left) $R$-modules into a corresponding statement about left (resp. right) $R^o$-modules in a completely natural way. In particular, $\underline{End}_R(P) = \underline{End}_{R^o}(P^o)$ as super rings.

One readily extends $\cdot^o$ to categories of bimodules as well, and checks that $\cdot^o: {}_S \mathfrak{M}_R \to {}_{R^o} \mathfrak{M}_{S^o}$ is an equivalence of categories as before.

Suppose $M \in \mathfrak{M}_R$ and $N \in {}_R \mathfrak{M}$. Using the universal property of the super tensor product (cf. \cite{DM}), the reader may verify that there is a natural isomorphism of superabelian groups $(M \otimes_R N)^o \to N^o \otimes_{R^o} M^o$ given by $m \otimes_R n \mapsto (-1)^{|m||n|} n \otimes_{R^o} m$. If $M$ is an $(S, R)$-bimodule and $N$ is an $(R, T)$-bimodule, then this isomorphism is an isomorphism in the category of $(T^o, S^o)$-bimodules.\\

\subsection{Projective modules.}

The definition of projective module is exactly the same as that in
the classical case:

\begin{defn}
A (right) $R$-module $P$ over a super ring $R$ is said to be {\it
projective} if $Hom_R(P, -): \mathfrak{M}_R \to (Ab)$ is an exact
functor; that is, if

\begin{equation*}
0 \to A \to B \to C \to 0
\end{equation*}\

\noindent is a short exact sequence of $R$-modules, then the induced
sequence of abelian groups:

\begin{equation*}
0 \to Hom(P, A) \to Hom_R(P, B) \to Hom_R(P, C) \to 0
\end{equation*}\

\noindent is short exact.
\end{defn}

\noindent {\bf Remark.} 1) $Hom(M, -)$ is a left exact functor for
any module $M$, that is, for an exact sequence:

\begin{align*}
0 \to A \to B \to C
\end{align*}\

the sequence

\begin{equation*}
0 \to Hom(P, A) \to Hom_R(P, B) \to Hom_R(P, C)
\end{equation*}\

\noindent is exact. The key property of a projective module $P$ is that if $B \to C$ is an epimorphism, $Hom(P, B) \to Hom(P, C)$ is also an epimorphism.\\

\noindent 2) Note that the definition of projective module involves
the categorical $Hom$ and not internal $\underline{Hom}$; as we
shall see, this situation is the precise converse to the definition
of generator, which involves internal $\underline{Hom}$ and not
categorical $Hom$.

An $R$-module $M$ is {\it free} if it has a homogeneous $R$-basis $\{e_i | f_j\}$ $i \in I$, $j \in J$, with $e_i$ even, $f_j$ odd. If $|I| = m$ and $|J| = n$ are both finite, then we define the {\it rank} of $M$ to be the pair of integers $m|n$. It is left to the reader to check that a free $R$-module satisfies the usual sort of universal property. Having made this definition, the proof of the following proposition is also identical to that in
the purely even case:

\begin{prop}
A right $R$-module $P$ is projective if and only if there exist a
free module $F$ and morphisms $\pi: F \to P$, $i: P \to F$ such that
$\pi \circ i = id_P$.
\end{prop}

We say that $i$ is a {\it splitting} and that $i$ {\it splits}
$\pi$. This proposition is equivalent to the assertion that $P$ is
projective iff $P$ is isomorphic to a direct summand of some free
module $F$, and shows that our definition of projective module is
entirely equivalent to that given in Appendix B of \cite{CCF}.

We have the following characterization of projective $R$-modules,
the super Dual Basis Lemma.

\begin{prop}\label{lem:dualbasis}
Let $R$ be an associative super ring. A right $R$-module $P$ is
projective if and only if there exist a family of homogeneous
elements $\{a_i , b_j : i \in I, j \in J\} \subseteq P$ ($a_i$ even,
$b_j$ odd) and homogeneous linear functionals $\{ f_i , g_j: i \in
I, j \in J\} \subseteq P^*$ ($f_i$ even, $g_j$ odd) such that, for
any homogeneous $c \in P$, $f_i(c) = 0, g_j(c) = 0$ for almost all
$i.j$, and $c = \sum_i a_i f_i(c) + \sum_j b_j g_j(c)$.
\end{prop}

\begin{proof}
Suppose such $a_i$'s, $b_j$'s and $f_i$'s, $g_j$'s exist. Consider
the epimorphism $h$ from the free module $F := \bigoplus d_i R \,
\oplus \, \bigoplus e_j R$ to $P$ defined by $h(d_i) = a_i$, $h(e_j)
= b_j$. (Here the $d_i$ are even, the $e_j$ odd). Then the map $k: P
\to F$ given by $h(c) = \sum_i d_i f_i(c) + \sum_j e_j g_j(c)$ is a
morphism splitting $h$, hence $P$ is projective.

Conversely, suppose $P$ is projective, and fix an epimorphism $h$
from a free module $\bigoplus d_i R \, \oplus \, \bigoplus e_j R$
onto $P$. Since $P$ is projective, there exists a splitting morphism
$k: P \to F$, and for homogeneous $c$, we may write $k(c) = \sum_i
d_i f_i(c) + \sum_j e_j g_j(c)$.

One checks that $f_i, g_j$ are $R$-linear and homogeneous, and that
$f_i$ and $g_i$ are zero for all but finitely many $i,j$.

Applying $h$ to both sides of the previous equation, we have $c =
\sum_i a_i f_i(c) + \sum_j b_j g_j(c)$, where $a_i := h(d_i), b_j :=
h(e_j)$.\\
\end{proof}

\noindent {\bf Remark.} The definition of a projective left
$R$-module is entirely analogous to that for right modules, and all the results
proven in this section are also true in the category of left
$R$-modules; the proofs may be easily supplied by the reader.\\

\subsection{Generators.}

Suppose $I \subseteq R$ is a homogeneous two-sided ideal in a super
ring $R$. Then the quotient ring $R/I$ is naturally a right (left)
$R$-module with grading induced from $R$, and the quotient map $R
\to R/I$ is an $R$-module epimorphism.

\begin{defn}
Let $P$ be a right $R$-module. The {\it trace ideal} of $P$, denoted
$tr(P)$, is the module:

\begin{align*}
tr(P) := \sum g(P)
\end{align*}\

\noindent where the sum is taken over all homogeneous $g \in P^* =
\underline{Hom}(P, R)$.
\end{defn}\

(We could just as well have left out the adjective ``homogeneous"
without affecting this definition at all, but what we have done
makes the proof of Prop. \ref{prop:generator} a little shorter.) One
checks readily that $tr(P)$ is a homogeneous two-sided ideal of $R$.
We now make the following important definition:

\begin{defn}
Let $P$ be a right $R$-module over an associative super ring $R$. We
say $P$ is a {\it generator} in $\mathfrak{M}_R$ if
$\underline{Hom}_R(P, -) : \mathfrak{M}_R \to (SAb)$ is a faithful
functor, i.e. for each $R$-morphism $f: M \to N, f_* = 0 \implies f
= 0$, where $f_*: \underline{Hom}_R(P, M) \to \underline{Hom}_R(P,
N)$ denotes the morphism in $(SAb)$ induced by composition with $f$.
\end{defn}\

\noindent {\bf Remark.} Here the super Morita theory diverges in a
crucial way from the ungraded theory: whereas in ordinary module
theory, categorical $Hom$ and internal $\underline{Hom}$ are the
same thing, here they are different, and one must take careful note that it is the
categorical $Hom$ that appears in the definition of projective module, while it is the internal $\underline{Hom}$ that appears in the definition of generator.\\




\subsection{Parity reversal.}

We recall some facts about the parity reversal functor $\Pi:
\mathfrak{M}_R \to \mathfrak{M}_R$ (resp. $\Pi: _R \mathfrak{M} \to
_R\mathfrak{M}$). Given a (left or right) $R$-module $M$, the
underlying set of $\Pi M$ is equal to the underlying set of $M$, but
endowed with the reverse grading:

\begin{align*}
&(\Pi M)_i = M_{i+1}
\end{align*}\\

for $i \in \mathbb{Z}_2$. $\Pi M$ is endowed with natural right (resp. left) $R$-module structures. The right $R$-module structure on $\Pi
M$ is defined to be the same as that on $M$:

\begin{align*}
m \cdot_\Pi r := m \cdot r
\end{align*}\

The left $R$-module structure is defined by:

\begin{align*}
r \cdot_\Pi m := (-1)^{|r|} r \cdot m
\end{align*}\

\noindent for homogeneous elements $r \in R, m \in \Pi M$. One checks that these indeed define $R$-module structures on $\Pi M$.

Although $\Pi$ is of course the identity when viewed purely as a map
of sets, given an element $m$ in the set $M$, we will occasionally
write $\Pi(m)$ for emphasis when we regard $m$ as an element of the
module $\Pi M$. With this convention, $\Pi(\Pi(m)) = m$, hence we
may write formally $\Pi^2 = id$ (we will justify this when we prove that $\Pi(\Pi M) = M$ as $R$-modules).

With this convention, the above definitions of the $R$-module structures may be rewritten as:

\begin{align*}
\Pi(r \cdot m) &= (-1)^{|r|} r \cdot_\Pi \Pi(m)\\
\Pi(m \cdot r) &= \Pi(m) \cdot_\Pi r
\end{align*}\

Given $f: M \to N$ a homogeneous $R$-homomorphism, we may define
associated homomorphisms $\Pi f: M \to \Pi N$ and $f \Pi: \Pi M \to
N$

\begin{align*}
&(\Pi f)(m) := \Pi(f(m))\\
&(f \Pi)(\Pi m) := f(m)
\end{align*}\




The reader will check that these are indeed $R$-module morphisms. In the next proposition, we collect some basic facts, which can be
mostly found in Ch.3, 1.5 of \cite{Man}, relating $\Pi$ to module homomorphisms.

\begin{prop}\label{prop:parity change}
Let $M$ and $N$ be right (left) $R$-modules. Then

\begin{itemize}

\item $\Pi(\Pi M) = M$ as right (left) $R$-modules.

\item There are odd isomorphisms of superabelian groups:

\begin{align*}
\Pi \cdot : \underline{Hom}_R(M, N) &\to \underline{Hom}_R(M, \Pi N)\\
&f \mapsto \Pi f\\
\cdot \Pi: \underline{Hom}_R(M, N) &\to \underline{Hom}_R(\Pi M, N)\\
&f \mapsto f \Pi
\end{align*}

\end{itemize}

\end{prop}

\begin{proof}

To prove that $\Pi (\Pi M) = M$, one first notes that the
gradings on $M$ and $\Pi (\Pi M)$ are the same:

\begin{align*}
[\Pi (\Pi M)]_i&= (\Pi M)_{i+1}\\
&=M_i
\end{align*}

Now we check they have the same $R$-module structures. The case of right modules is trivial, so we only need check the case of left modules.

\begin{align*}
r \cdot_{\Pi \Pi} m &= (-1)^{|r|} r \cdot_{\Pi} m\\
&= r \cdot m
\end{align*}

Thus $\Pi(\Pi M) = M$ in ${}_R \mathfrak{M}$ and $\mathfrak{M}_R$,
as desired.

We now discuss the second item. Let $f \in \underline{Hom}(M, N)$ be
homogeneous. Then $\Pi f \in \underline{Hom}(M, \Pi N)$ is
homogeneous of the opposite parity. Hence $f \to \Pi f$ is an odd
map. That it is a homomorphism of abelian groups is clear.
Similarly, if $g: M \to \Pi N$, then $g \Pi: M \to \Pi(\Pi N) = N$
is a homomorphism of the opposite parity, hence $g \mapsto \Pi g$
gives an odd homomorphism of abelian groups $\underline{Hom}(M, \Pi
N) \to \underline{Hom}(M, N)$ and it is left to the reader to check
that it is inverse to $f \mapsto \Pi f$.

The proof for $\cdot \Pi$ proceeds similarly and is also left to the
reader.






\end{proof}

We now give several equivalent characterizations of a generator:

\begin{prop}\label{prop:generator}
Let $P$ be a right $R$-module over an associative super ring $R$.
The following are equivalent:\\

\begin{enumerate}
\item $P$ is a generator in $\mathfrak{M}_R$.
\item $tr(P) = R$.
\item $R$ is a direct summand of a finite direct sum $\oplus_i P \oplus_j \Pi P$.
\item $R$ is a direct summand of a direct sum $\oplus_i P \oplus_j \Pi P$.
\item Every $M \in \mathfrak{M}_R$ is an epimorphic image of some direct sum $\oplus_i P \oplus_j \Pi P$.\\
\end{enumerate}
\end{prop}

\begin{proof}
1) $\implies$ 2). Suppose $I := tr(P) \neq R$. Then the quotient map
$R \to R/I$ is nonzero in $\mathfrak{M}_R$, hence by the hypothesis
that $P$ is a generator, there is some $g \in \underline{Hom}_R(P,
R)$ such that $P \xrightarrow{g}R \to R/I$ is nonzero. But then
$g(P) \nsubseteq I$, contradicting the definition of $g$.

\noindent 2) $\implies$ 3). By 2), there exist $f_1, \dotsc, f_m,
g_1, \dotsc, g_n$, with $f_i$ even and $g_j$ odd, such that $\sum_i
f_i (p_i) + \sum g_j(q_j) = 1$. By taking homogeneous components of
this equation, we may assume that the $p_i$ are even, $q_j$ are odd.
Then $(f_1, \dotsc, f_m, g_1 \Pi, \dotsc, g_n \Pi): P \oplus \dotsc
\oplus P \oplus \Pi P \oplus \dotsc \oplus \Pi P \to R$ is a split
epimorphism, with splitting given by $1 \mapsto (p_1, \dotsc, p_m,
q_1, \dotsc, q_n)$, whence 3).

\noindent 3) $\implies$ 4). Tautological.

\noindent 4) $\implies$ 5). Follows easily, since $M$ is an
epimorphic image of a free module.

\noindent 5) $\implies$ 1). Suppose $f: M \to N$ is a nonzero
morphism. By 5) there exists an epimorphism $\oplus_i P \oplus_j \Pi
P \to M$. The composition $\oplus_i P \oplus_j \Pi P \to M
\xrightarrow{f} N$ is clearly nonzero. Hence, either for some $i$,
$P_i = P \xrightarrow{g} M \xrightarrow{f} N$ is nonzero, or for
some $j$, $\Pi P_j = \Pi P \xrightarrow{h} M \xrightarrow{f} N$ is
nonzero. But by Prop. \ref{prop:parity change}, $h: \Pi P \to M$ may
be regarded as a homomorphism $h \Pi: \Pi(\Pi P) = P \to M$ of the
opposite parity, and the composition $P \xrightarrow{h \Pi} M
\xrightarrow{f} N$ is nonzero. In either case, we have proven that
$\underline{Hom}_R(P, -)$ is faithful, as desired.

\end{proof}

\noindent {\bf Remark.} Just as Lam points out in \cite{Lam} for the
classical case, the notions of finitely generated projective module
and generator for a super ring are complementary: $P$ is finitely
generated projective iff $P$ is a direct summand of $R^{m|n}$ for
some $m,n$; $P$ is a generator iff $R$
(regarded naturally as a free $R$ module of rank $1|0$ with basis
$\{1\}$) is a direct summand of $P^{m|n} = P^m \oplus (\Pi P)^n$ for
some $m,n$. This suggests that combining the two conditions will
yield an interesting notion:

\begin{defn}
An $R$-module $P$ is a {\it progenerator} iff it is a finitely
generated projective generator.
\end{defn}

Just as in the ungraded case, it is the concept of progenerator
which is crucial to the theory of Morita equivalences.

We will need to show that being a progenerator is a categorical
property. More precisely:

\begin{prop}\label{prop:progenerator categorical}
Let $R$, $S$ be super rings, and $F: \mathfrak{M}_R \to \mathfrak{M}_S$ an equivalence of
categories. If $P$ is a progenerator in $\mathfrak{M}_R$, then
$F(P)$ is a progenerator in $\mathfrak{M}_S$.
\end{prop}

\begin{proof}
Suppose $P$ is a progenerator in $\mathfrak{M}$. The statement that $P$ is projective is equivalent to:\\

\noindent {\it The functor $Hom_R(P, -)$ is exact.}\\

Since $F$ is an equivalence of categories, it takes short exact
sequences in $\mathfrak{M}_R$ to short exact sequences in
$\mathfrak{M}_S$, hence $F(P)$ is projective if $P$ is projective.

One checks that the statement that $P$ is finitely generated is equivalent to:\\

\noindent {\it For any family of submodules $\{N_i : i \in I\}$ of
$P$, if $\sum_{i \in I} N_i = M$, then $\sum_{i \in J} N_i = M$ for
some finite subset $J \subseteq I$.}\\

This is clearly preserved by a category equivalence, so $F(P)$ is
finitely generated if $P$ is finitely generated. Finally, the statement that $P$ is a generator is:\\

\noindent {\it The functor $\underline{Hom}_R(P, -)$ is faithful.}\\

This is again preserved by category equivalences, so $F(P)$ is a
generator if $P$ is a generator.
\end{proof}

\noindent {\bf Remark.} As in the previous section, all definitions and theorems that were stated for right $R$-modules also hold for categories ${}_R \mathfrak{M}$ of left $R$-modules. Proofs are left to the reader.\\

\section{Super Morita theory}

\subsection{The super Morita context.}

Let $R$ be a super ring, $P$ a right $R$-module, and $Q := P^* =
\underline{Hom}_R(P, R)$, $S := \underline{End}_R(P)$, with $Q, S$
both acting on the left on $P$. Thus $P$ becomes an $(S,
R)$-bimodule.

As in the classical case, we define the left action of $R$ on $Q$ by
$(rq)(p) := r(qp)$ (``$RQP$-associativity") and the right action of
$S$ by $(qs)(p) := q(sp)$ (``$QSP$-associativity"), thus making $Q$
an $(R, S)$ bimodule in a natural way.

There are also several important pairings involving $Q$ and $P$.

\begin{lem}
Let $p, p' \in P, q, q' \in Q$. Define the pairings:

\begin{align*}
&Q \times P \to R\\
&(q, p) \mapsto qp := q(p)\\
\\
&P \times Q \to S\\
&(p, q) \mapsto pq\\
\end{align*}

\noindent where $(pq)(p') := p(qp')$. Then:\\

\noindent 1) $(q, p) \mapsto qp$ defines an $(R, R)$-morphism
$\alpha: Q \otimes_S
P \to R$;\\
\noindent 2) $(p, q) \mapsto pq$ defines an $(S, S)$-morphism
$\beta: P
\otimes_R Q \to S$.\\
\end{lem}

\begin{proof}
As $P$ is an $(S, R)$-bimodule and $Q$ an $(R, S)$-bimodule, the
tensor product $Q \otimes_S P$ makes sense and is an $(R,
R)$-bimodule. $QSP$-associativity implies the $S$-bilinearity $(qs,
p) = (q, sp)$ of the first pairing, hence it induces a morphism (of
superabelian groups) $\alpha: Q \otimes_S P \to R$.

Now $\alpha$ is given by $q \otimes p \mapsto qp$. Since we have $|q
\otimes p| = |q| + |p| = |qp|$, we see $\alpha: Q \otimes_S P \to R$
preserves parity. That $\alpha$ is actually an $(R,R)$-morphism is a
consequence of $RQP$-associativity and $QPR$-associativity,
respectively. We have already discussed $RQP$-associativity $r(qp) =
(rq)p$, and $QPR$-associativity $(qp)r = q(pr)$ is a restatement of
the $R$-linearity of $q \in P^*$.

Similarly, the proof of 2) is an easy consequence of $PRQ$-, $SPQ$-,
and $PQS$-associativities, and is left to the reader. The only thing
to note is that $p \otimes q \mapsto pq$ is a parity preserving map
(the argument is the same as that given for $\alpha$), hence $\beta$
is a morphism of superabelian groups.
\end{proof}

\begin{defn}
The {\it super Morita context} associated to $P_R$ is the
6-tuple

\begin{equation*}
(R, P, Q, S; \alpha, \beta).
\end{equation*}
\end{defn}

\noindent {\bf Remark.} In order to carry out this discussion for left $R$-modules, we have to make the following conventions. If $P \in {}_R \mathfrak{M}$, then $S = \underline{End}({}_R P)$ still acts on the left. We convert the left $S$-action into a right $S^o$-action in the canonical way:
\begin{align*}
ps := (-1)^{|s||p|} s(p)
\end{align*}
so that $P$ becomes an $(R, S^o)$-bimodule. Homomorphisms in $Q =
\underline{Hom}(P, R)$ now act on $P$ on the right, as discussed in section 2. Then $Q$ becomes an $(S^o, R)$-bimodule, with the right
action of $R$ on $Q$ given by right multiplication in $R$, and
the left action of $S^o$ given by:

\begin{align*}
p(sq) := (ps)q\\
\end{align*}

In what follows, we could reformulate and reprove everything that we do for categories of right modules for categories of left modules. This is a lengthy exercise, but it requires nothing new except a change of notation, and in the author's view, little would be gained in so doing. (The reader is invited to provide these proofs to his own satisfaction). Instead, we use the category equivalence $\cdot^o$ to turn statements about left $R$-modules into statements about right $R^o$-modules, therefore reducing the results for left modules to the already-proven case of right modules.\\

We shall now proceed to prove some important facts about the super
Morita context.

\begin{prop}\label{prop:moritacontext1} Let $P \in \mathfrak{M}_R$, $(R, P, Q, S; \alpha, \beta)$ the super Morita context associated to $P$. Then:

\begin{enumerate}
\item $P$ is a generator iff $\alpha$ is onto.\\
\item Suppose $P$ is a generator. Then\\

\begin{enumerate}[a)]

\item $\alpha: Q \otimes_S P \to R$ is an $(R,R)$ isomorphism.\
\item $Q \cong \underline{Hom}_S(P, S)$ as $(R,S)$-bimodules.\
\item $P \cong \underline{Hom}_S(Q, S)$ as $(S,R)$-bimodules.\
\item $R \cong \underline{End}_S({}_S P)^o \cong \underline{End}_S(Q_S)$
as super rings.\\
\end{enumerate}
\end{enumerate}
\end{prop}

\begin{proof}
1) follows from Prop. \ref{prop:generator}. For 2), suppose $P_R$ is
a generator, then we have an equation $1 = \sum_i q_i p_i + \sum_j
\tilde{q}_j \tilde{p}_j$ where $q_i \in Q, p_i \in P$ are even and
$\tilde{q}_j \in Q, \tilde{p}_j \in P$ are odd. For 2a), suppose
$\sum_k q'_k \otimes p'_k$ is homogeneous, and that $\alpha(\sum_k
q'_k \otimes p'_k) = \sum_k q'_k p'_k = 0$. Then

\begin{align*}
\sum_k q'_k \otimes p'_k &= \sum_{k, i, j} (q_i p_i + \tilde{q}_j \tilde{p}_j) q'_k \otimes p'_k\\
&= \sum_{k, i, j} \left[ q_i (p_i q'_k) + \tilde{q}_j (\tilde{p}_j q'_k) \right] \otimes p'_k\\
&= \sum_{i, k} \left[ q_i \otimes (p_i q'_k) p'_k \right] + \sum_{j,k} \tilde{q}_j \otimes (\tilde{p}_j q'_k) p'_k\\
&= \sum_i q_i \otimes p_i \left( \sum_k q'_k p'_k \right) + \sum_j \tilde{q}_j \otimes \tilde{p}_j \left( \sum_k q'_k p'_k \right)\\
&=0\\
\end{align*}

To prove 2b), we define a map $\lambda: Q \to \underline{Hom}_S({}_S
P, {}_S S)$ by $p \, \cdot \lambda(q) := (-1)^{|p||q|} pq \in S$. That $\lambda(q)$ is parity-preserving is clear: $| p \cdot
\lambda(q)| = |pq| = |p| + |q|$. That $\lambda(q) \in
\underline{Hom}_S(P, S)$ is a consequence of $SPQ$-associativity
$(sp)q = s(pq)$:

\begin{align*}
(sp) \cdot \lambda(q) &= (-1)^{|sp||q|}(sp)q\\
&=(-1)^{(|s|+|p|)|q|} (sp)q\\
&=(-1)^{(|s|+|p|)|q|} s(pq)\\
&=(-1)^{|s||q|} s(p \cdot \lambda(q))\\
&=(-1)^{|s||\lambda(q)|} s (p \cdot \lambda(q))\\
\end{align*}

Hence $\lambda$ is an even homomorphism. We now show that $\lambda$
is injective: suppose $pq = 0$ for all $p \in P$. Then, since $1_R =
\sum_i q_i p_i + \sum_j \tilde{q}_j \tilde{p}_j$, we have:

\begin{align*}
q &= 1_R q\\
&= [\sum_i q_i p_i + \sum_j \tilde{q}_j \tilde{p}_j]q\\
&= \sum_i q_i (p_i q) + \sum_j \tilde{q}_j (\tilde{p}_j q)\\
&= \sum_i q_i (0) + \sum_j \tilde{q}_j (0)\\
&=0\\
\end{align*}

We now prove that $\lambda$ is surjective: suppose $f \in
\underline{Hom}_S(P, S)$. We have:

\begin{align*}
pf &= [p \left(\sum_i q_i p_i + \sum_j \tilde{q}_j \tilde{p}_j \right)] f\\
&= \sum_i ((pq_i)p_i)f + \sum_j ((p\tilde{q}_j) \tilde{p}_j)f\\
&= \sum_i (-1)^{|pq_i||f|}(pq_i)(p_i f) + \sum_j (-1)^{|p \tilde{q}_j||f|} (p\tilde{q}_j)(\tilde{p}_j f)\\
&= p \left( \sum_i (-1)^{|pq_i||f|} q_i (p_i f) + \sum_j (-1)^{(|p \tilde{q}_j||f|} \tilde{q}_j (\tilde{p}_j f) \right)\\
&= p \left( \sum_i (-1)^{|p||f|} q_i (p_i f) + \sum_j (-1)^{(|p|+1)|f|} \tilde{q}_j (\tilde{p}_j f) \right)\\
\end{align*}

\noindent Hence $f = \lambda \left( \sum_i q_i (p_i f) + (-1)^{|f|}
\sum_j \tilde{q}_j (\tilde{p}_j f) \right)$, and so
$\lambda$ is an isomorphism. This proves 2b).

We define super ring homomorphisms

\begin{align*}
\sigma: R \to \underline{End}({}_S P)^o \hspace{5mm} \text{ and} \hspace{5mm}  \tau: R \to \underline{End}(Q_S)\\
\end{align*}

\noindent by $p \sigma(r) := (-1)^{|p||r|} pr$ and $\tau(r)q := rq$.

The proof that $\sigma(r) \in \underline{End}({}_S P)$ is just like
the proof for $\lambda$; the proof that $\tau \in
\underline{End}(Q_S)$ is trivial.

Both $\sigma$ and $\tau$ preserve parity:

\begin{align*}
|p \sigma(r)| &= |(-1)^{|p||r|} pr|  \\
&= |p| + |r|\\
|\tau(r)q| &= |rq|\\
&= |q| + |r|\\
\end{align*}

The proof that $\tau$ is a super ring homomorphism is a triviality.
For the case of $\sigma$, we only note that since we are allowing
$\underline{End}({}_S P)$ to act on the right instead of the left,
we really have a super ring homomorphism from $R$ to
$\underline{End}({}_S P)^o$, as the reader will easily verify.

Hence $\sigma$ and $\tau$ are super ring morphisms. The proof that
$\sigma, \tau$ are isomorphisms is similar to the proof that
$\lambda$ is an isomorphism and is left to the reader.
\end{proof}

The following proposition is complementary to the one just proved;
it applies to finitely generated projective modules.

\begin{prop}\label{prop:moritacontext2}
Let $P \in \mathfrak{M}_R$, $(R, P, Q, S; \alpha, \beta)$ the super Morita context associated to $P$. Then:\\

\begin{enumerate}
\item $P$ is a finitely generated projective module iff $\beta$ is onto.\\
\item Suppose $P$ is a finitely generated projective module. Then:\\

\begin{enumerate}[a)]
\item $\beta: P \otimes_R Q \to S$ is an $(S,S)$ isomorphism.\
\item $Q \cong \underline{Hom}_R(P_R, R_R)$ as $(R,S)$-bimodules.\
\item $P \cong \underline{Hom}_R({}_R Q, {}_R R)$ as $(S,R)$-bimodules.\
\item $S \cong \underline{End}(P_R) \cong \underline{End}({}_R Q)^o$
as super rings.
\end{enumerate}
\end{enumerate}
\end{prop}

\begin{proof}

1) $\beta$ is onto iff there is an equation $1_S = \sum_l p''_l
q''_l + \sum_m \tilde{p}''_m \tilde{q}''_m$.

The super Dual Basis Lemma (Prop. \ref{lem:dualbasis}) states that
this is completely equivalent to $P$ being finitely generated
projective. The proof of 2) is completely analogous to the proof of
the previous proposition, using the equation $1_S = \sum_l p''_l
q''_l + \sum_m \tilde{p}''_m \tilde{q}''_m$.

The only things that need to be noted are the following: the
homomorphism $\lambda': P \to \underline{Hom}_R({}_R Q, {}_R R)$ is
given by $q \lambda'(p) := (-1)^{|q||p|} qp$, which is clearly
parity preserving.

The map $Q \to \underline{Hom}_R(P_R, R_R)$ is just the identity,
hence a parity preserving homomorphism. We also need to define ring
homomorphisms:

\begin{align*}
\sigma': S \to \underline{End}(P_R) \hspace{5mm} \text{ and} \hspace{5mm} \tau': S \to \underline{End}({}_R Q)^o
\end{align*}

\noindent which are parity-preserving. This is obvious for $\sigma' := id$. We define $\tau'$ by $q \cdot \tau'(s) :=
(-1)^{|q||s|} qs$, which is clearly parity-preserving. The proof that $\tau'$ is an isomorphism proceeds as before.
\end{proof}

\subsection{The super Morita theorems.}

Finally, we come to the main results of this note. The first
(``super Morita I") states that tensoring with an $R$-progenerator $P$
defines an equivalence of categories between $\mathfrak{M}_R$ (resp.
${}_R \mathfrak{M}$) and $\mathfrak{M}_S)$ (resp ${}_S
\mathfrak{M}$, where $S = \underline{End}_R(P)$.

\begin{thm}\label{thm:moritaI}
Let $R$ be a super ring, $P_R$ a progenerator, and $(R, P, Q, S;
\alpha, \beta)$ the super Morita context associated with $P_R$.
Then:

\begin{enumerate}

\item $ - \otimes_R Q: \mathfrak{M}_R \to \mathfrak{M}_S$ and $-
\otimes_S P: \mathfrak{M}_S \to \mathfrak{M}_R$ are mutually
inverse category equivalences.

\item $P \otimes_R -: {}_R \mathfrak{M} \to {}_S \mathfrak{M}$
and $Q \otimes_S -: {}_S \mathfrak{M} \to {}_R \mathfrak{M}$ are
mutually inverse category equivalences.

\end{enumerate}

\end{thm}

\begin{proof}

Let $M$ be a right $R$-module. Then

\begin{align*}
(M \otimes_R Q) \otimes_S P &\cong M \otimes_R (Q \otimes_S P)\\
&\cong M \otimes_R R\\
&\cong M\\
\end{align*}

\noindent where the first isomorphism is the canonical associativity
isomorphism of the super tensor product induced by $(m \otimes q)
\otimes p \mapsto m \otimes (q \otimes p)$, the second isomorphism
follows from Prop. \ref{prop:moritacontext1}, and the third is the
canonical isomorphism induced by $m \otimes r \mapsto mr$. All of
these isomorphisms are clearly functorial in $M$, whence it follows
that the composition of $- \otimes_S P$ with $ - \otimes_R Q$ is
naturally equivalent to the identity.

If $N$ is a right $S$-module, the same proof goes through, switching
the roles of $R$ and $S$, as well as those of $P$ and $Q$, and
instead of Prop. \ref{prop:moritacontext1}, we invoke Prop.
\ref{prop:moritacontext2} to conclude that $N \otimes_S (P \otimes_R
Q) \cong N \otimes_S S$.

Again, all isomorphisms are functorial in $M$, whence it follows
that the composition of $- \otimes_R Q$ with $ - \otimes_S P$ is
naturally equivalent to the identity. Hence we have proven that $P
\otimes_R -$ and $Q \otimes_S -$ are mutually inverse category
equivalences.

The proof of part 2) is completely analogous and is left to the
reader.\end{proof}

We now make a useful definition.

\begin{defn}
Let $A, B$ be super rings. An $(A, B)$-bimodule $C$ is {\it
faithfully balanced} if the natural maps $A \to
\underline{End}(C_B)$ and $B \to \underline{End}({}_A C)^o$ are both
isomorphisms of super rings.
\end{defn}

By Props. \ref{prop:moritacontext1} and \ref{prop:moritacontext2},
if $P_R$ is a progenerator, ${}_S P_R$ and ${}_R Q_S$ are faithfully
balanced bimodules.

The following proposition shows that the roles of $P$ and $Q$ in the
Morita theory are completely symmetric.

\begin{prop}\label{prop:leftright}
Suppose $P_R$ is a progenerator. Then ${}_S P, {}_R Q, Q_S$ are also
progenerators, and $\alpha, \beta$ are isomorphisms.
\end{prop}

\begin{proof}
First, we prove that $Q_S$ is a progenerator. By Prop.
\ref{prop:moritacontext1}(2), we have $\underline{Hom}_S(Q_S, S_S)
\cong P$ and $\underline{End}(Q_S) \cong R$. As $\alpha, \beta$ are
surjective, we see from Prop. \ref{prop:moritacontext1}(1) and
\ref{prop:moritacontext2}(1), applied to $Q_S$, that $Q_S$ is a
progenerator.

The proof for ${}_R Q$ is analogous to that for ${}_S P$ and so we only do the case of ${}_S P$. We denote $\underline{End}_S({}_S P)$ by $S'$. $P$ is an $(S, S'^o)$-bimodule. By 2d of Prop. \ref{prop:moritacontext1} we have a super ring isomorphism $\sigma: R \cong (S')^o$. Recalling that $\sigma$ is defined by the right action of $R$, the natural equivalence of categories $\sigma^{-1}: \mathfrak{M}_{(S')^o} \to \mathfrak{M}_R$ induced by the isomorphism $\sigma^{-1}$ clearly sends $P_{(S')^o}$ to $P_R$. By hypothesis $P_R$ is an $R$-progenerator, hence $P_{(S')^o}$ is an $(S')^o$-progenerator. The $S$-dual $P^\vee := \underline{Hom}_S({}_S P, {}_S S)$ is then an $((S')^o, S)$-bimodule in the usual way, and again the equivalence of categories $\sigma^{-1}: {}_{(S')^o} \mathfrak{M} \to {}_R \mathfrak{M}$ sends ${}_{(S')^o} P^\vee$ to ${}_R P^\vee$.

We want to show the pairings $\alpha': P \otimes_{(S')^o} P^\vee \to S$ and $\beta': P^\vee \otimes_S P \to (S')^o$ are epimorphisms. Identifying $\mathfrak{M}_{(S')^o}$ with $\mathfrak{M}_R$ and ${}_{(S')^o} \mathfrak{M}$ with ${}_R \mathfrak{M}$, and identifying the $(R,S)$-modules $Q$ and $P^\vee$ via 2c of Prop. \ref{prop:moritacontext1} we see that this is equivalent to $\beta: P \otimes_R Q \to S$ and $\alpha: Q \otimes_S P \to R$ being epimorphisms. Since $P_R$ is a progenerator, that is indeed the case.

Now we apply the functor $\cdot^o$ to convert everything to right modules. ${}_S P_{(S')^o}$ becomes ${}_{S'} P_{S^o}$, and since $P^\dag := \underline{Hom}_{S^o}(P_{S^o}, S^o_{S^o}) = P^\vee$, ${}_{(S')^o} P^\vee_S$ becomes ${}_S P^\dag_{(S')^o}$. The $(S, S)$ (resp. $((S')^o, (S')^o)$) epimorphisms $\alpha': P \otimes_{(S')^o} P^\vee \to S$ and $\beta': P^\vee \otimes_S P \to (S')^o$ become $(S^o, S^o)$ (resp. $(S', S')$) epimorphisms $\alpha': P \otimes_{(S')^o} P^\vee \to S$ and $\beta': P^\vee \otimes_S P \to (S')^o$. By \ref{prop:moritacontext1} and \ref{prop:moritacontext2}, $P_{S^o}$ is a progenerator. Since being a progenerator is a categorical property (Prop. \ref{prop:progenerator categorical}) we finally conclude that ${}_S P$ is a progenerator as well.
\end{proof}

Hence if one of $P_R, {}_R Q, {}_S P, Q_S$ are $R$-progenerators
(resp. $S$-progenerators), the rest of them are too.

Now we prove the second of the main Morita theorems (``super Morita II"), a converse to
the first. It states that {\it every} Morita equivalence between two
categories of super modules is (up to natural equivalence) of the form given
in the first Morita theorem: i.e., by tensoring with a
progenerator.

\begin{thm}
Let $R$, $S$ be two super rings, and

\begin{align*}
&F: \mathfrak{M}_R \to \mathfrak{M}_S\\
&G: \mathfrak{M}_S \to \mathfrak{M}_R\
\end{align*}

\noindent be mutually inverse category equivalences. Let $Q = F(R_R)$, $P =
G(S_S)$. Then there are natural bimodule structures $P = {}_S P_R, Q
= {}_R Q_S$, which yield functor isomorphisms $F \cong - \otimes_R
Q$ and $G \cong - \otimes_S P$.
\end{thm}

\begin{proof}
Since $G(S) = P$, $\underline{End}(S_S) \cong \underline{End}(P_R)$ as super rings, via $f \mapsto G(f)$. It is easily seen that $\underline{End}(S_S) \cong S$ as super rings via the morphism $f \mapsto f(1)$ (the inverse being given by $s \mapsto g_s$, where $g_s(x) = sx$). Hence we have a natural $(S, R)$-bimodule structure ${}_S P_R$.
Via a similar argument for $Q$, we have a natural $(R, S)$-bimodule structure ${}_R Q_S$. As $S_S$ is a progenerator in $\mathfrak{M}_S$, $P_R$ is a
progenerator in $\mathfrak{M}_R$ by Prop. \ref{prop:progenerator
categorical}.

$\underline{Hom}_R(P, R)$ has its usual left $R$-module structure induced by left multiplication in $R$, and a right $S$-module structure defined by $QSP$-associativity. We verify that $\underline{Hom}_R(P, R) \cong Q$ as right $S$-modules:

\begin{align*}
\underline{Hom}_R(P, R) \cong \underline{Hom}_S(F(P), F(R)) \cong \underline{Hom}_S(S_S, Q_S) \cong Q
\end{align*}\\

Hence, the super Morita context associated to $P_R$ is $(R, P, Q, S;
\alpha, \beta)$, where $\alpha, \beta$ are the pairings defined
previously. Now the first Morita theorem applies; it remains to show
that $F$ is naturally equivalent to the functor $- \otimes_R Q$.

Given any $M \in \mathfrak{M}_R$,

\begin{align*}
F(M) \cong \underline{Hom}_S(S_S, F(M)) \cong \underline{Hom}_R(P_R, M_R)
\end{align*}\\

\noindent whence $F \cong \underline{Hom}_R(P_R, -) \cong -
\otimes_R Q$ by Lem. \ref{lem:rhomlemma}. Similarly, we have $G \cong
\underline{Hom}_S(Q_S, -) \cong - \otimes_S P$.

\end{proof}

We have one more theorem (``super Morita III") which characterizes the
isomorphism classes of equivalences between super module categories.
In order to state this theorem, we require the following:

\begin{defn}
Let $R, S$ be super rings. An $(S, R)$-bimodule ${}_S P_R$ is an {\it $(S,
R)$-progenerator} if ${}_S P_R$ is faithfully balanced and $P_R$ is
an $R$-progenerator.
\end{defn}

We may now state the third of our Morita theorems:

\begin{thm}
Let $R$ and $S$ be two super rings. Then the isomorphism classes of
category equivalences $\mathfrak{M}_S \to \mathfrak{M}_R$ are in
one-to-one correspondence with the isomorphism classes of $(S,
R)$-progenerators. Composition of category equivalences corresponds
to tensor products of these progenerators.
\end{thm}

\begin{proof}
Each $(S, R)$-progenerator yields a category equivalence $- \otimes_S P: \mathfrak{M}_S \to \mathfrak{M}_R$. The isomorphism class of this equivalence depends only on the isomorphism class of the $(S, R)$-bimodule $P$. Conversely, suppose $G: \mathfrak{M}_S \to \mathfrak{M}_R$ is a category equivalence. Then $P := G(S_S)$ is an $(S,R)$-progenerator, as in the proof of the second Morita theorem. The isomorphism class of the $(S, R)$-bimodule $P$ clearly depends only on the isomorphism class of the equivalence $G$, proving the first statement. If ${}_R P'_T$ is an $(R, T)$-progenerator, the composition of equivalences $\mathfrak{M}_S \to \mathfrak{M}_R \to \mathfrak{M}_T$ is given by $- \otimes_S (P \otimes_R P')$, whence the second conclusion.\\
\end{proof}

\noindent {\bf Remark:} One can state versions of Morita II and Morita III for categories of left modules, and prove them in exactly the same way as we have done for categories of right modules, or one can use $\cdot^o$ to reduce to the right-module versions of the theorems. This is left to the reader.\\

\section{Application: super Azumaya algebras}
\subsection{Super Azumaya algebras.}

Let $k$ be a field. Recall that a ungraded $k$-algebra $B$ is said
to be {\it central} if its center is $k$, and {\it simple} if $B$
has no non-trivial two-sided ideals. Those algebras which are
finite-dimensional and central simple over $k$ are characterized by
the Artin-Wedderburn theorem:

\begin{thm}
Let $A$ be a algebra over a field $k$ which is finite dimensional as
a $k$-vector space. Then $A$ is central simple over $k$ if and only
if the map:

\begin{align*}
&A \otimes A^o \to End_k(A)\\
&a \otimes b \mapsto (x \mapsto axb)
\end{align*}

\noindent is an isomorphism of $k$-algebras.
\end{thm}

This theorem generalizes to the context of superalgebras. For this
we must recall some basic definitions.

The supercenter of a super ring $A$ is the sub-super ring:

\begin{align*}
Z(A) := \{x \in A : ax = (-1)^{|a||x|} xa \text{ for all homogeneous $a \in A$}
\end{align*}\

From now on, we suppose that $R$ is a supercommutative ring, i.e.
$rr' = (-1)^{|r||r'|} r'r$ for all homogeneous $r, r' \in R$. An
$R$-superalgebra is a super ring $A$ with a super ring morphism $i: R \to A$
such that $i(R) \subseteq Z(A)$.

The opposite $A^o$ is also an $R$-superalgebra in a natural way; since $i(R) \subseteq Z(A)$, and $R$ is supercommutative, $i^o: R \to A^o$ is a super ring morphism (here $i^o$ denotes the map $R \to A^o$ that agrees with $i$ as a map of sets).

If $A$, $B$ are $R$-superalgebras, the tensor product $A \otimes_R B$ possesses a natural structure of
$R$-superalgebra, the multiplication being given by:

\begin{equation*}
(a \otimes b) \cdot (c \otimes d) := (-1)^{|b||c|} ac \otimes bd
\end{equation*}\

Given any $R$-superalgebra $A$, there is a natural $R$-superalgebra
morphism $\phi: A \otimes_R A^o \to \underline{End}_R(A)$ given by:

\begin{equation*}
a \otimes b \mapsto (x \mapsto (-1)^{|b||x|}axb)
\end{equation*}\

From now on, we will denote the superalgebra $A
\otimes_R A^o$ by $A^e$ for the sake of brevity.

\noindent We say that a $k$-superalgebra $A$ is {\it central} if its
supercenter equals $k$, and $A$ is {\it simple} if $A$ has no
non-trivial two-sided homogeneous ideals. The super Artin-Wedderburn
theorem (see, e.g. \cite{Var}) then states that:

\begin{thm}
Let $A$ be a superalgebra over a field $k$, $char(k) \neq 2$, which
is finite dimensional as a $k$-super vector space. Then $A$ is
central simple over $k$ if and only if $\phi: A^e \to
\underline{End}_k(A)$ is an isomorphism of $k$-superalgebras.
\end{thm}

In ordinary commutative algebra, the notion of central simple
algebra over a field $k$ has been generalized to the category of
algebras over a commutative ring by adopting the conclusion of the
Artin-Wedderburn theorem as a definition. The resulting objects are
called {\it Azumaya algebras} (one may see \cite{KnO} for more on
the basics of ungraded Azumaya algebras).

We define the corresponding super notion as follows.

\begin{defn}
Let $A/R$ be a superalgebra over a supercommutative ring $R$. We say
that $A$ is a {\it super Azumaya algebra} over $R$ iff $A$ is a
faithful, finitely generated projective $R$-module, and the natural
morphism $\phi: A^e \to \underline{End}_R(A)$ is an isomorphism of
$R$-superalgebras.
\end{defn}

\noindent {\bf Example.} Let $k$ be an algebraically closed field of characteristic $\neq 2$. We define the {super skew field} (\cite{DM}) to be:

\begin{align*}
\mathbb{D} := k[\theta], \theta \text{ odd, }, \theta^2 = -1
\end{align*}\\

In \cite{DM} it is shown that $\mathbb{D}$ is central simple over
$k$, hence a super Azumaya algebra over $k$, and that (the Brauer
equivalence class of) $\mathbb{D}$ generates the super Brauer group
of $k$. For the definition of the super Brauer group of a field and
basic results, see again \cite{DM}; the original source for this
material is \cite{Wal}.

By definition a super Azumaya algebra $A$ is a progenerator, hence,
by the results of the previous section, there is a Morita equivalence between ${}_R \mathfrak{M}$ and
${}_{A^e} \mathfrak{M})$. Our aim in this section is to make this Morita equivalence even more
explicit by expressing it in terms of the notion of {\it supercommutant}.

We note that the concepts of $(A, A)$-bimodule and left $A^e$-module
are completely equivalent: if $M$ is a left $A^e$ module, we define
an $(A, A)$-bimodule structure by:

\begin{align*}
a  \cdot m &:= (a \otimes 1) \cdot m\\
m \cdot a &:= (-1)^{|a||x|}(1 \otimes a) \cdot m
\end{align*}\

where $a \in A$, $m \in M$ are homogeneous.

Conversely, if $N$ is an $(A, A)$-bimodule, we define a left
$A^e$-module structure on $N$ via:

\begin{align*}
(a \otimes b) \cdot n := (-1)^{|b||n|} a \cdot n \cdot b
\end{align*}\






\noindent It's readily seen that these correspondences are compatible with $(A, A)$-morphisms and $A^e$-morphisms respectively. Hence, there is a natural equivalence of categories between the category of $(A,A)$-bimodules and that of left $A^e$-modules.

We begin with a super version of a standard fact from the theory of
modules (cf. Exercise 20 of \cite{Lam}, Ch. 2):

\begin{lem}\label{lem:homlemma}
Let $R$ be a super ring (not necessarily commutative), and $P, B \in \mathfrak{M}_R$. Define
the morphism of superabelian groups $\sigma_{P,B}: P^* \otimes_R B
\to \underline{Hom}_R(P, B)$ by:

\begin{align*}
 [\sigma_{P,B}(f \otimes b)](x) := (-1)^{|x||b|} f(x) \cdot b \\
\end{align*}

Then if $P$ is a finitely generated projective $R$-module,
$\sigma_{P,B}$ is an isomorphism. Furthermore, $\sigma_{P, -}$ is
functorial in $B$: given a $R$-morphism $k: B \to B'$, the diagram of superabelian groups:

\begin{equation*}
\xymatrix{P^* \otimes _R B \ar[r]^-{\sigma_{P, B}} \ar[d]^{id \otimes k} & \underline{Hom}_R(P, B) \ar[d]^{k_*}\\
P^* \otimes _R B' \ar[r]^-{\sigma_{P, B'}} & \underline{Hom}_R(P, B')}\\
\end{equation*}

\noindent is commutative.\\

\end{lem}

\begin{proof} First, one readily checks that $\sigma_{P,B}$ is a parity-preserving homomorphism of abelian groups, so that it is indeed an $SAb$-morphism. First we consider the case where $P$ is free of finite rank, beginning with the case where $P$ is free of rank $1|0$ or $0|1$. So suppose that $P = R$ as $R$-modules. Then one may check readily (as in the ungraded case)
that $\underline{Hom}_R(R, B) \cong B$ in $SAb$, via $\phi \mapsto \phi(1)$. By composing this isomorphism with $\sigma_{R^*, B}$, we have an $SAb$-morphism $\sigma'_{R^*, B}:
R^* \otimes_R B \to B$, with $\sigma'_{R^*, B}(f \otimes b) := f(1)b$.

We are now reduced to proving that $\sigma'_{R^*,B}$ is an
isomorphism; this follows after one checks that the inverse
homomorphism $\sigma'^{-1}_{R^*, B}$ is given by $b \mapsto 1^*
\otimes b$, where $1^*$ is the functional in $R^*$ dual to $1 \in
R$.

The case of $P = \Pi R$ is completely analogous, but one must keep
careful track of the parity reversals.

We have an $SAb$ isomorphism $\underline{Hom}(\Pi R, B)
\cong \Pi B$ via $\phi \mapsto \Pi[\phi \Pi(\Pi^2 1)] = \Pi[\phi(\Pi(1))]$. This is the composition of three isomorphisms: the odd isomorphism $\underline{Hom}_R(\Pi R, B) \to \underline{Hom}_R(\Pi^2 R, B) = \underline{Hom}_R(R, B)$ given by $\phi \mapsto \phi \Pi$ (see Lem \ref{prop:parity change}), the above (even) isomorphism $Hom_R(R, B) \cong B$, and finally the odd isomorphism $\Pi(id): B \to \Pi B$.

Composing this with $\sigma_{(\Pi
R)^*, B}$, we have a new morphism $\sigma'_{(\Pi R)^*, B}: (\Pi R)^*
\otimes B \to \Pi B$ with $\sigma'_{(\Pi R)^*, B}(f \otimes b) = (-1)^{|b|}\Pi[f(\Pi
1)b]$, and it is again enough to show that $\sigma'_{(\Pi R)^*, B}$
is an isomorphism. One checks that the inverse homomorphism
$\sigma'^{-1}_{(\Pi R)^*, B}$ is given by $\Pi b \mapsto (\Pi (1))^*
\otimes b$.

Now let $P$ be a finite-rank free module: $P = \bigoplus_j P_j$,
where there are only finitely many $j$, and each $P_j$ is either
isomorphic to $R$ or $\Pi R$. We use the standard fact that
$\underline{Hom}$ is compatible with direct sums:

\begin{align*}
\underline{Hom}(\bigoplus_j P_j, M) \cong \prod_j \underline{Hom}(P_j,
M) \cong \bigoplus_j \underline{Hom}(P_j, M)
\end{align*}

The first isomorphism is the universal property of the direct sum, hence is natural. The second isomorphism is the natural identification of a finite direct
product with a finite direct sum (here we use the hypothesis that $F$ has finite rank).

One may check that $\sigma_{-, B}$ is compatible with these identifications, in the
sense that:

\begin{align*}
\xymatrix@!C{P_j^* \otimes _R B \ar[r]^-{\sigma_{P, B}|_{_{P_j}}} & \underline{Hom}_R(P_j, B) \\
P^* \otimes _R B \ar[r]^-{\sigma_{P, B}} \ar[u]^{i_j^* \otimes id} & \underline{Hom}_R(P, B) \ar[u]^{i_j^*}}\\
\end{align*}

\noindent where $i_j$ is the inclusion of the $j$th summand into the
direct sum. Hence, taking the direct sum over all (finitely many)
indices $j$, we have:

\begin{align*}
\xymatrix@!C{\bigoplus_j (P_j^* \otimes _R B) \ar[r]^-{\bigoplus_j \sigma_{P, B}|_{_{P_j}}} & \bigoplus_j \underline{Hom}_R(P_j, B) \\
P^* \otimes _R B \ar[r]^-{\sigma_{P, B}} \ar[u]^{\bigoplus_j i_j^* \otimes id} & \underline{Hom}_R(P, B) \ar[u]^{\bigoplus_j i_j^*}}\\
\end{align*}

For any $j$, we have either $P_j \cong R$ or $P_j \cong \Pi R$. By
what we proved earlier, $\sigma_{P, B}|_{_{P_j}}$ is an isomorphism
for each $j$, thus their direct sum $\bigoplus_j \sigma_{P,
B}|_{_{P_j}}$ is also an isomorphism. The morphisms $\bigoplus_j (
i_j^* \otimes id)$ and $\bigoplus_j i^*_j$ are just the identity
maps, hence isomorphisms. By commutativity of the diagram, it
follows that $\sigma_{P, B}$ is an isomorphism as well.

In turn, we reduce the general case to the case of a finite-rank
free module as follows. Since $P$ is finitely generated projective,
there exists a free module $F$ of finite rank and a split
epimorphism $\pi: F \to P$ with splitting $i: P \to F$. As
$i \bullet \pi = id_P$, so $i^* \circ \pi^* = id_{P^*}$. Hence
$\pi^*$ is injective, $i^*$ surjective.

Let $\underline{\pi}^*$ denote the morphism from
$\underline{Hom}_R(P,B)$ to $\underline{Hom}_R(F,B)$ induced by
$\pi$, and $\underline{i}^*$ the morphism $\underline{Hom}_R(F,B)$
to $\underline{Hom}_R(P,B)$ induced by $i$. Then the same
considerations as before give $\underline{i}^* \circ
\underline{\pi}^* = id_{\underline{Hom}_R(P,B)}$. Consequently,
$\underline{\pi}^*$ is injective, $\underline{i}^*$ surjective.

We claim the following diagrams are commutative:

\begin{equation*}
\xymatrix{F^* \otimes _R B \ar[r]^-{\sigma_{F, B}} & \underline{Hom}_R(F, B)\\
P^* \otimes _R B \ar[u]_{\pi^* \otimes id_B} \ar[r]^-{\sigma_{P, B}} & \underline{Hom}_R(P, B) \ar[u]_{\underline{\pi}^*}}\\
\end{equation*}

\begin{equation*}
\xymatrix{F^* \otimes _R B \ar[r]^-{\sigma_{F, B}} \ar[d]^{i^* \otimes id_B} & \underline{Hom}_R(F, B) \ar[d]^{\underline{i}^*}\\
P^* \otimes _R B \ar[r]^-{\sigma_{P, B}} & \underline{Hom}_R(P, B)}\\
\end{equation*}

We will check commutativity of the first. Let $y \in F, f \otimes b
\in P^* \otimes B$. Then:

\begin{align*}
[\underline{\pi}^* \circ \sigma_{P,B}(f \otimes b)](y) &= [\sigma_{P,B}(f \otimes b)](\pi(y))\\
&= (-1)^{|\pi(y)||b|}f(\pi(y)) \cdot b\\
&= (-1)^{|y||b|} \pi^*(f)(y) \cdot b\\
&= [\sigma_{F,B}(\pi^*(f) \otimes b)](y)\\
&= [(\sigma_{F,B} \circ (\pi^* \otimes id))(f \otimes b)](y)\\
\end{align*}

\noindent which is the statement that the first diagram commutes.

Commutativity of the second diagram is analogous and is left to the reader.




The lemma follows easily from this, combined with the fact that
$\sigma_{F,B}$ is an isomorphism ($F$ is a free module). For injectivity
of $\sigma_{P,B}$: suppose $a \in P^* \otimes B$ and
$\sigma_{P,B}(a) = 0$. But commutativity of the first diagram gives
$(\pi^* \otimes id_B)(a) = \sigma^{-1}_{F,B} \circ \underline{\pi}^*
\circ \sigma_{P,B}(a)$ hence $(\pi^* \otimes id_B)(a) = 0$. Since
$\pi^* \otimes id_B$ is injective, $a = 0$.

For surjectivity of $\sigma_{P,B}$: suppose $c \in
\underline{Hom}_R(P, B)$. Then since $\underline{i}^*$ is onto, $c =
\underline{i}^*(c')$ for some $c' \in \underline{Hom}_R(F,B)$. Let
$a = i^* \circ \sigma^{-1}_{F,B}(c')$. Then by commutativity of the
second diagram, $\sigma_{P,B}(a) = c$.

It remains to prove functoriality in $B$. Let $k: B \to B'$ be a
morphism, and let $x \in P$.
\begin{align*}
\ [k_* \circ \sigma_{P,B}(f \otimes b))](x) &= k[\sigma_{P,B}(f \otimes b)(x)]\\
&=k((-1)^{|b||x|} f(x) \cdot b)\\
&=(-1)^{|k(b)||x|} f(x) \cdot k(b)\\
&=[\sigma_{P, B'}(f \otimes k(b))](x)\\
&=[(\sigma_{P, B'} \circ (id \otimes k))(f \otimes b)](x)
\end{align*}\\

\noindent which is the statement that the given diagram commutes.
\end{proof}

\noindent {\bf Remark.} We will need the following additional property of $\sigma_{P, B}$ for our applications. If $P$ is also a right $S$-module for some super ring $S$, $\underline{Hom}_R(P, B)$ has a natural structure of left $S$-module via the ``pullback" action:

\begin{align*}
(sf)(x) := (-1)^{|f(x)||s|} f(xs)
\end{align*}\\

$P^*$ also has a left $S$-module structure by the same formula, which induces a left $S$-module structure on $P^* \otimes_R B$. It is easily seen that with these left $S$-module structures, $\sigma_{P, B}$ is a morphism (and by the lemma, an isomorphism) in ${}_S \mathfrak{M}$.

Of course, we have an analogous result for right $R$-modules. Since we need this for the proof of ``super Morita II", we formulate it explicitly:

\begin{lem}\label{lem:rhomlemma}
Let $R$ be a super ring (not necessarily commutative), and $P, B \in {}_R \mathfrak{M}$. Define
the morphism of superabelian groups $\sigma_{B,P}: B \otimes_R P^*
\to \underline{Hom}_R(P, B)$ by:

\begin{align*}
 [\sigma_{B,P}(b \otimes f)](x) := b \cdot f(x) \\
\end{align*}

Then if $P$ is a finitely generated projective $R$-module,
$\sigma_{B,P}$ is an isomorphism. Furthermore, $\sigma_{-, P}$ is
functorial in $B$: given a $R$-morphism $k: B \to B'$, the diagram of superabelian groups:

\begin{equation*}
\xymatrix{P^* \otimes _R B \ar[r]^-{\sigma_{B, P}} \ar[d]^{id \otimes k} & \underline{Hom}_R(P, B) \ar[d]^{k_*}\\
P^* \otimes _R B' \ar[r]^-{\sigma_{B', P}} & \underline{Hom}_R(P, B')}\\
\end{equation*}

\noindent is commutative.\\
\end{lem}

To prove this, one can, as usual, either adapt the proof of Lem. \ref{lem:homlemma} to the category of right modules, or use the functor $\cdot^o$ to reduce everything to the case of left modules.\\

\subsection{The supercommutant.}

\begin{defn} Let $M$ be an $(A,A)$-bimodule. The {\it supercommutant}
of $M$ is the $(R, R)$-bimodule $M^A$ generated by:

\begin{equation*}
\{m \in M : a m = (-1)^{|a||m|} m a \; \forall \text{ homogeneous } a \in A \}.
\end{equation*}\
\end{defn}

\noindent Equivalently, interpreting $M$ as a left $A^e$-module, we
see that $M^A$ may be defined in terms of the $A^e$-action, as the
$(R, R)$-bimodule generated by:

\begin{equation*}
\{m \in M : (a \otimes 1) m = (1 \otimes a) m \; \forall  \text{ homogeneous } a \in A \}
\end{equation*}\

$M^A$ so defined is indeed an $(R, R)$-bimodule: as $i(R) \subseteq Z(A)$, the action of $R$ on $M$ commutes with the action of $A$. Let $a \in A$, $r \in R$, and suppose $m \in M^A$. Then:

\begin{align*}
a (r m) &= (ar) m\\
&= (-1)^{|a||r|} (ra) m\\
&= (-1)^{|a||r|+|ra||m|} m (ra)\\
&= (-1)^{|a||r|+|ra||m|} (mr) a\\
&= (-1)^{|ra||m|} (rm) a
\end{align*}

\noindent which is the statement that $rm \in M^A$. That $mr \in
M^A$ may be checked in completely analogous fashion.

We have the following functoriality property: if $f: M \to N$ is an
$A^e$-morphism, then clearly $f(M^A) \subseteq N^A$ and $f' :=
f|_{M^A}: M^A \to N^A$ is an $(R, R)$-morphism. Clearly $(f \circ g)' =
f' \circ g'$ and $(id_M)' = id_{M^A}$.

It follows that the operation of taking the supercommutant may be regarded as
a functor $-^A: {}_{A^e}\mathfrak{M} \to {}_R\mathfrak{M}$, where $M
\mapsto M^A, f \mapsto f'$ for any $A^e$-module $M$ and
$A^e$-morphism $f: M \to N$.

We now show

\begin{thm}\label{thm:commutant}
$\underline{Hom}_{A^e}(A, M)$ is naturally isomorphic to $M^A$ as a left
$R$-module.
\end{thm}

\begin{proof} We define the isomorphism $F; \underline{Hom}_{A^e}(A, M) \to M^A$ as follows. Suppose $f:A \to M$ is an $A^e$-homomorphism. Then:

\begin{align*}
f(a) &= f((a \otimes 1) \cdot 1)\\
&=a \cdot f(1) \cdot 1\\
&=a \cdot f(1)\\
\\
f(a) &= f((1 \otimes a) \cdot 1)\\
&=(-1)^{|f(1)||a|} \cdot f(1) \cdot a\\
&= (-1)^{|f(1)||a|} f(1) \cdot a\\
\end{align*}

\noindent hence $f(1) \in M^A$. Thus we have a parity-preserving map:

\begin{align*}
F: &\underline{Hom}_{A^e}(A, M) \to M^A\\
&f \mapsto f(1)\\
\end{align*}

\noindent Recalling that the $R$-module structure on $\underline{Hom}_{A^e}(A, M)$ is defined by the ``pullback" action:

\begin{align*}
(rf)(x) := (-1)^{|r||f(x)|} f(xr)
\end{align*}\


\noindent it is readily checked that $F$ is an $R$-morphism. Conversely, suppose $m \in M^A$. Then we define a map $g_m: A \to M$
by:

\begin{align*}
g_m(x) = x \cdot m\\
\end{align*}

\noindent We check that $g_m$ so defined is indeed an $A^e$-homomorphism:

\begin{align*}
g_m(a \otimes b \cdot x) &= g((-1)^{|b||x|} axb)\\
&= (-1)^{|b||x|} axbm\\
&= (-1)^{|b|(|x| + |m|)} axmb   \, \text{ (since $m \in M^{A^e}$)}\\
&= (-1)^{|b||g_m(x)|} \, a \cdot g_m(x) \cdot b\\
&=(a \otimes b) \cdot g_m(x)\\
\end{align*}

\noindent Hence $m \mapsto g_m$ is a parity-preserving map $G: M^A \to
\underline{Hom}_{A^e}(A, M)$, and it's easily seen that $G$ is
inverse to $F$. Since $F$ is an $R$-morphism, so is $G$.

We now verify the naturality statement of the theorem: that $F: \underline{Hom}_{A^e}(A, M) \to M^A$ and $G: M^A \to
\underline{Hom}_{A^e}(A, M)$ are functorial in $M$.

Let $h: M \to N$ be an $A^e$-morphism. Then $h$ induces the
$R$-morphism $h_*: \underline{Hom}_{A^e}(A, M) \to
\underline{Hom}_{A^e}(A, N)$ by $h_*(f) = h \circ f$.

The statement that $F$ is a natural transformation is the equality $F \circ h_* = h' \circ F$, where $h'$
denotes the restriction of $h$ to $M^A$.

\begin{align*}
F \circ h_*(f) &= F \circ (h \circ f)\\
&= (h \circ f)(1)\\
&=h(f(1))\\
&=h' \circ F(f)\\
\end{align*}

The proof that $G$ is a natural transformation is completely analogous: we verify the equality $G \circ h' = h_* \circ G$.

\begin{align*}
G \circ h'(m) &= G(h(m))\\
&=g_{h(m)}\\
&=h \circ g_m\\
&=h_* \circ G(m)\\
\end{align*}

\end{proof}

\subsection{The main result.} Putting all this together, we have our main result:

\begin{thm}
Let $R$ be a supercommutative ring, $A/R$ a super Azumaya algebra. Then the functors:

\begin{align*}
A \otimes_R - &:   {}_R \mathfrak{M} \to {}_{A^e} \mathfrak{M}\\
-^A &: {}_{A^e} \mathfrak{M} \to {}_R \mathfrak{M}\\
\end{align*}

\noindent are mutually inverse category equivalences.

\end{thm}

\begin{proof}
By ``super Morita I" (Thm. \ref{thm:moritaI}), we have that $A \otimes_R -:
$ and $A^* \otimes_{A^e} -$ are mutually inverse category
equivalences (after composing with the obvious category equivalence
${}_S\mathfrak{M} \to {}_{A^e} \mathfrak{M}$ induced by the
isomorphism $\phi: A^e \to \underline{End}_R(A) = S$). Let $A^\vee$ denote
$\underline{Hom}_{A^e}(A, A^e)$. We claim there is a sequence of
natural isomorphisms:

\begin{align*}
A^* \otimes_{A^e} M &\cong A^\vee \otimes_{A^e} M\\
&\cong \underline{Hom}_{A^e}(A, M)\\
&\cong M^A\\
\end{align*}

By part 2b) of Lem. \ref{prop:moritacontext2} $A^* \cong
A^\vee$ as an $(R,A^e)$-bimodule; hence the first isomorphism exists and is obviously functorial in $M$. Note that the $R$-action on $A^\vee$ is given by $x(rf) := (-1)^{|r||f(x)|} (xr)f$.
By Lemma \ref{prop:leftright}, $A^\vee$ is a projective
$A^e$-module. Hence, by Lemma \ref{lem:homlemma} and the following Remark, the second
isomorphism (of left $R$-modules) exists and is functorial in $M$. By Thm.
\ref{thm:commutant}, the third isomorphism (of left $R$-modules) exists and is also
functorial in $M$. We conclude that the identification of $A^*
\otimes_{A^e} M$ with $M^A$ is functorial in $M$, hence we have
shown that the functor $-^A$ is naturally equivalent to $A^*
\otimes_{A^e} -$, so is also a functor inverse to $A \otimes_R -$,
as desired.

\end{proof}

\section{Acknowledgments} Thanks go to V.S. Varadarajan, who
generously provided much helpful advice and support. I am also
indebted to Ben Antieau, who read a draft of the paper and gave many suggestions for improvement. Finally, the germ of this work was suggested in a note of P.
Deligne \cite{Del}, who pointed out the relationship of the super
skew field $\mathbb{D}$ to $\Pi$-invertible sheaves.

\bibliographystyle{amsplain}
\bibliography{moritabib}

\end{document}